\documentclass[11pt]{article}
\usepackage{amsmath,amsthm,amsfonts,amssymb,amscd}
\numberwithin{equation}{section}
\usepackage{tikz}
\usetikzlibrary{decorations.markings}
\usetikzlibrary{arrows.meta,positioning,calc}
\usepackage[scr]{rsfso}
\usepackage{stackrel}
 \usepackage{amsmath}
\usepackage{enumerate} 
\usepackage{color}
\usepackage{xcolor}
\usepackage{float}
\usepackage{soul}
\usepackage{graphicx}
\usepackage{mathpazo}
\usepackage{fancyhdr}
\usepackage{empheq}
\usepackage{hyperref}
\hypersetup{
  colorlinks=true,
  linkcolor=red,
  citecolor=blue,
  urlcolor=blue
}
\usepackage[margin=1in]{geometry}

\def\XX{\mathbb{X}}
\def\YY{\mathbb{Y}}

\def\Bar{\overline}
\def\ra{\rangle}
\def\la{\langle}
\def\ve{\varepsilon}
\def\B{\mathbb{B}}
\def\h{\hfill\Box}
\def\R{\mathbb{R}}
\def\argmin{ \mathop{{\rm argmin}}}
\def\ox{\bar{x}}

\def\cone{\mbox{\rm cone}\,}

\def\Im{\mbox{\rm Im}\,}

\def\span{\mbox{\rm span}\,}

\def\dom{\mbox{\rm dom}\,}
\def\Ker{\mbox{\rm Ker}\,}

\def\supp{\mbox{\rm supp}\,}

\def\sign{\mbox{\rm sign}\,}

\def\cl*co{\mbox{\rm cl}^*\mbox{\rm co}\,}

\def\cl{\mbox{\rm cl}\,}

\def\h{\hfill\triangle}
\def\dn{\downarrow}
\def\O{\Omega}

\def\oR{\Bar{\R}}

\def\Lm{\Lambda}
\def\N{\mathbb{N}}

\def\hs7{\hspace*{7pt}}

\renewcommand{\theequation}{\thesection.\arabic{equation}}

\def\h{\hfill\Box}

\begin{document}

\newtheorem{Theorem}{Theorem}[section]
\newtheorem{Conjecture}[Theorem]{Conjecture}
\newtheorem{Proposition}[Theorem]{Proposition}
\newtheorem{Remark}[Theorem]{Remark}
\newtheorem{Lemma}[Theorem]{Lemma}
\newtheorem{Corollary}[Theorem]{Corollary}
\newtheorem{Definition}[Theorem]{Definition}
\newtheorem{Example}[Theorem]{Example}
\newtheorem{Fact}[Theorem]{Fact}
\newtheorem*{pf}{Proof}
\renewcommand{\theequation}{\thesection.\arabic{equation}}
\usetikzlibrary{decorations.markings}

\usetikzlibrary{decorations.markings}

\tikzset{
  crossed top/.style={
    postaction={
      decorate,
      decoration={
        markings,
        mark=at position 0.62 with {
          \draw[black, thick] (-2pt,6pt) -- (2pt,-6pt);
        }
      }
    }
  },
  crossed bottom/.style={
    postaction={
      decorate,
      decoration={
        markings,
        mark=at position 0.38 with {
          \draw[black, thick] (-2pt,6pt) -- (2pt,-6pt);
        }
      }
    }
  }
}
\normalsize
\normalfont
\medskip
\def\endproof{$\h$\vspace*{0.1in}}

\title{\bf Isolated calmness of regularized linear inverse problems}
\date{}

\author{
Tran T. A. Nghia\thanks{Department of Mathematics and Statistics, Oakland University, Rochester, MI 48309, USA. Emails: nttran@oakland.edu; hnpham@oakland.edu; nghiavo@oakland.edu.}
\and
Huy N. Pham\footnotemark[1]
\and
Nghia V. Vo\footnotemark[1]
\and
Khoa V. H. Vu\thanks{Department of Electrical and Computer Engineering, Oakland University, Rochester, MI 48309, USA. Email: khoavu@oakland.edu.}
}
\maketitle
\begin{center}
\emph{Dedicated to Professor Phan Quoc Khanh on the occasion of his 80th birthday}
\end{center}

\vspace{0.1in}

{\small \noindent {\bf Abstract.}
This paper studies the isolated calmness property of solution mappings arising from convex regularized linear inverse problems. We mainly  establish a tangent-cone characterization of this property. For convex piecewise linear-quadratic regularizers, the isolated calmness coincides with the solution uniqueness, leading to simple verification procedures. In particular, for analysis sparsity regularization, our condition can be verified by linear programming.}

\section{Introduction}
\label{sec:intro}

Recovering a signal $x_0$ in a Euclidean space $\XX$ from its observation $b_0=\Phi x_0$ in another Euclidean space $\YY$ over a linear operator $\Phi:\XX\to \YY$ is a fundamental linear inverse problem.
Solving the linear system of equations
\[
\Phi x = b_0
\]
may fail to recover the exact signal $x_0$ in general; especially in scenario where the system is ill-posed, leading to infinitely many solutions. To restore well-posedness, typical approaches incorporates prior low-complexity information, such as sparsity, group sparsity, bounded variation, or low rank, through adding a regularizer $R:\XX\to \R$ as an objective function and solving the optimization problem
\begin{equation}\label{prob:exact}
  \min_{x\in \XX}\quad  R(x)
  \quad \mbox{subject to} \quad
  \Phi x=b_0,
\end{equation}
is a successful technique widely adopted in retrieving $x_0$ \cite{CT05,DET05,CR09,FR13,GSH11,BB18,VPF15}. While the squared Euclidean norm for the regularizer leads to the classical minimum-norm solution and connects to Tikhonov--Morozov regularization \cite{T43,TA77,M93,EHN96}, many modern applications in compressed sensing, image processing, and machine learning, call for nonsmooth regularizers (e.g., the $\ell_1$ norm, the $\ell_1/\ell_2$ norm, total variation seminorms, and the nuclear norm) whose properties promote prior low-complexity structures; see, e.g.,   \cite{BDE09,BB18,CR09,CP11,FR13} for the influence of  this regularization technique in different areas.

In this paper, we study the parameterized regularized linear inverse problem
\begin{equation}\label{eq:parameterized_problem_intro}
P(b):\qquad
\min_{x\in \XX} \quad R(x)
\quad \mbox{subject to} \quad
\Phi x=b,
\end{equation}
where $b\in \YY$ is the only parameter. Its corresponding  solution mapping is 
\[
S(b):= \argmin\,\{R(x)|\, \Phi (x)=b \}.
\]
For a reference solution $x_0\in S(b_0)$, a first natural question is whether $x_0$ is the unique solution of $P(b_0)$, i.e., whether the noiseless constrained model achieves exact recovery; see, e.g., \cite{F05,Tr06,ZYC15,G17,ZYY16,FNT23,HKS23}. However, solution uniqueness itself does not quantify how the recovered solutions behave when the observed data $b$ is inexact. Our paper mainly studies the Lipschitz-type continuity of the solution mapping $S$ {\em at} $b_0$ {\em for} $x_0$, which is usually referred to as the \emph{isolated calmness} \cite[Section~3.9]{DR14} of $S$ at $b_0$ for $x_0$, in the sense that there exist a neighborhood $U$ of $x_0$ in $\XX$, a neighborhood $V$ of $b_0$ in $\YY$, and a constant $\kappa>0$ such that
\begin{equation}\label{eq:isolated_calm_intro}
S(b)\cap U
\subset
x_0+\kappa \|b-b_0\|\mathbb B_{\XX}
\quad \mbox{for all } b\in V,
\end{equation}
where $\mathbb B_{\XX}$ is the Euclidean unit ball in $\XX$. This property means that any optimal solution of $P(b)$ with small perturbation $b$ of $b_0$ in a neighborhood of $x_0$ can be approximated by $x_0$ with a linear rate $\|b-b_0\|$. It plays an important role in stability theory for optimization problems and convergence analysis of algorithms; see, e.g., \cite{DR14, I17}. The mapping $S$ does not need to be single-valued around $b_0$, but $S(b_0)$ must consist of only $x_0$. This notion is closely related to stable recovery theory \cite{DET05,G11,GSH11,FPV13,H13,NPV25} that considers a different parameterized unconstrained optimization problem
(aka the {\em Tikhonov regularization problem} \cite{TA77})
\begin{equation}\label{p:Lass}
 \min_{x\in \XX}\quad  \frac{1}{2}\|\Phi x-b\|^2+\mu R(x) 
\end{equation} 
with two parameters: the tuning or regularization parameter $\mu>0$  and the noise parameter $b$. Stable recovery means any solution $x_{\mu,b}$ of this problem can be estimated linearly by $x_0$ in the sense that 
\begin{equation}\label{eq:SC}
 \|x_{\mu,b}-x_0\|\le \mathcal{O}(\delta),
\end{equation}
when $\mu$ is proportional to $\delta$ and $\|b-b_0\|\le \delta$ with $\delta$ being the noise level. For many polyhedral regularizers, solution uniqueness is often sufficient for stable recovery \cite{GSH11,G17,ZYC15}, while beyond  polyhedral settings  solution uniqueness is not enough to guarantee stable recovery and the isolated calmness; see, e.g., \cite[Example~3.4]{NPV25} and our Example~\ref{ex:NoSR}. \cite[Theorem~3.3]{NPV25} shows that stable recovery occurs if and only if
\[
    \Ker \Phi\cap T_{\partial R^*({\rm Im}\, \Phi^*)}(x_0)=\{0\}, 
\]
where $T_{\partial R^*({\rm Im}\, \Phi^*)}(x_0)$ is the {\em tangent cone} to the (possibly non-convex) set $\partial R^*({\rm Im}\, \Phi^*)$ at $x_0$. In this paper, we prove that this condition also characterizes the isolated calmness property \eqref{eq:isolated_calm_intro}.
\vspace{0.1in}

\noindent {\bf Our contributions.}
\begin{itemize}
\item We characterize solution uniqueness for composite regularizers of the form $R(x)=h(Kx),$
where $K:\XX\to \mathbb{E}$ is a linear operator between two Euclidean spaces and $h: \mathbb{E}\to\mathbb R$ is a continuous convex function. 
Assuming that $x_0$ is an optimal solution of $P(b_0)$, we prove that $x_0$ is the unique solution of $P(b_0)$ if and only if
\begin{equation}\label{eq:radial_cone_uniqueness_intro}
\operatorname{Ker}\Phi
\cap
K^{-1}
\left(
\operatorname{cone}\big(\partial h^*(z)-Kx_0\big)
\right)
=
\{0\},
\end{equation}
for some $z \in \Lambda(x_0)$, the {\em pre-dual certificate set} defined later in \eqref{def:DC}. This condition extends recent geometric uniqueness characterization for convex optimization problems \cite{FNP25} for the composite setting.

\item We then establish a geometric characterization of isolated calmness for the solution mapping $S$. Under the standing continuity assumption on $h$ and a mild full-row-rank assumption on $\Phi$, imposed only to ensure the non-emptiness of $S(b)$, we prove that $S$ has the isolated calmness property at $b_0$ for $x_0$ if and only if
\begin{equation}
\operatorname{Ker}\Phi
\cap
T_{\partial R^*(\operatorname{Im}\Phi^*)}(x_0)
=
\{0\}.
\label{eq:tangent_cone_ic_intro}
\end{equation}
%Here $T_{\partial R^*(\operatorname{Im}\Phi^*)}(x_0)$ denotes the contingent cone to the set $\partial R^*(\operatorname{Im}\Phi^*)$ at $x_0$. 
This condition is consistent with the second-order geometric structure underlying the stable recovery theory developed in \cite{NPV25}.

%\item For {\em convex piecewise linear--quadratic regularizers}, we prove that isolated calmness of $S$ at $b_0$ for $x_0$ is equivalent to solution uniqueness. Consequently, the radial-cone condition in \eqref{eq:radial_cone_uniqueness_intro} becomes a practical certificate for isolated calmness. This class contains many regularizers used in optimization, statistics, and inverse problems, including piecewise-linear penalties such as the $\ell_1$ norm, anisotropic total variation, and fused Lasso, as well as piecewise linear--quadratic penalties such as the elastic net, Huber-type penalties, and Blake--Zisserman-type regularizers \cite{ZH05,H73,BZ87}.

\item We make the preceding criterion explicit for analysis $\ell_1$ regularization, $R(x)=\|Kx\|_1.$ In this case, after choosing a pre-dual certificate, the relevant radial cone can be described using the active index set of $Kx_0$ and the saturated index set of the certificate. This gives a finite-dimensional conic feasibility condition that can be checked via linear programming.

\item Finally, we show that the equivalence between local single-valuedness and isolated calmness fails beyond the piecewise linear--quadratic setting by constructing a nonpolyhedral example involving the $\ell_1/\ell_2$ norm for which the solution mapping is locally single-valued around the reference solution, but isolated calmness fails. %The counterexample shows that qualitative stability does not necessarily imply Lipschitz-type stability.

\end{itemize}

\noindent
\textbf{Outline of the paper.}
Section~\ref{sec:prelim} recalls the notation and basic tools from convex analysis used throughout the paper. Section~\ref{sec:main} studies the solution mapping associated with the regularized inverse problem, including the radial cone characterization of solution uniqueness and the tangent cone characterization of isolated calmness. We then discuss consequences for convex piecewise linear--quadratic regularizers, specialize the criteria to analysis $\ell_1$ regularization, and present a non-polyhedral example separating local single-valuedness from isolated calmness of the solution mapping.

\section{Preliminaries}\label{sec:prelim}

In this paper, we denote by $\XX$ a Euclidean space endowed with the inner product $\la u,v\ra$ for $u,v \in \XX$ and $$\|x\|:=\sqrt{\langle x,x\rangle}\quad \mbox{for}\quad x\in \XX$$ is the corresponding Euclidean norm. The notations $\B_{r}(\ox)$ and $\B$, respectively, denote the closed ball with center $\ox\in\XX$ and radius $r>0$ and the closed unit ball centered at the origin in $\XX.$ For a linear operator $\Phi:\XX\to\YY$ between two Euclidean spaces, we use $\Im \Phi\subset \YY$ and $\Ker \Phi\subset \XX$ to denote its {\em range} and {\em kernel space}. Denote by $\Phi^*:\YY\to \XX$ the {\em adjoint} operator of $\Phi$. Let $\R_+$ and $\R_{++}$ be the sets of nonnegative and strictly positive real numbers, respectively.

Let us present several notions and classical results from convex analysis that will be used throughout the paper; see, e.g., \cite{R70}. Let $\Omega$ be a closed convex set in $\XX$ and $\ox \in \Omega$. We denote $\cone \Omega$ for the \emph{conic hull} of $\Omega$, ${\rm aff}\, \Omega$ for the $\emph{affine hull}$ of $\Omega$, and $\span \Omega$ for the $\emph{span}$ of $\Omega$ . The \emph{relative interior} of $\Omega$ is defined by
\[
{\rm ri}\, \Omega :=\{x\in \Omega|\;\exists\, \ve>0, \B_\ve(x)\cap  {\rm aff }\, \Omega \subset \Omega\}.
\]
The \emph{cone of feasible directions} or \emph{radial cone} to $\Omega$ at $\ox\in \Omega$ is defined by
\begin{equation}\label{RT}
    \mathcal{R}_\Omega(\ox) :={\rm cone}\,(\Omega-\ox)=\R_+(\Omega-\ox).
\end{equation}
For a nonempty set $\Omega\subset \XX$, the \emph{asymptotic cone} or \emph{recession cone} is defined by
\begin{equation}\label{AC}
\Omega^\infty
:=
\left\{
x\in \XX |\,
\exists\, \{x_k\} \subset \Omega,\ \lambda_k \small\downarrow0,\ 
\text{with } \lambda_k x_k \to x \text{ as } k\to \infty
\right\}.
\end{equation}
Given a nonempty, convex cone $C\subset \XX$, its {\em polar cone} is defined as 
\[
C^{\circ}:= \{v\in \XX|\, \la v,c\ra \leq 0 \text{ for all } c \in C\}.
\]

\begin{Definition}[Tangent/Contingent cone]\label{def:tan}
Let $\O$ be a closed (possibly non-convex) subset of $\XX$. The tangent/contingent cone to $\O$ at a point $\bar{x}\in \O$ is defined by
\begin{equation}\label{eq:tan}
        T_{\O}(\bar{x}):=\{w\in \XX|\, \exists\, t_k\downarrow 0, w_k\to w: \bar{x}+t_kw_k\in \O \quad \forall\, k\in \N\}.
    \end{equation}
\end{Definition}

Let $\varphi:\XX\to \oR :=\R\cup\{\infty\}$ be a lower semicontinuous (l.s.c.) convex function with nonempty domain
\[
\dom\varphi:=\{x\in\XX\mid \varphi(x)<\infty\}.
\]
For any $\ox\in\dom\varphi$, the (Fenchel) {\em convex subdifferential} of $\varphi$ at $\ox$ is defined as
\begin{equation}\label{defi:conv-subd}
\partial\varphi(\ox):=\{v\in\XX\mid \varphi(x)-\varphi(\ox)\ge \la v,x-\ox\ra\quad \forall\, x\in\XX\}.
\end{equation}
Associated with $\varphi$, its (Legendre--Fenchel) {\em conjugate} is the l.s.c. convex function $\varphi^*:\XX\to\oR$ given by
\[
\varphi^*(v):=\sup\{\la v,x\ra-\varphi(x)\mid x\in\XX\}
\quad\mbox{for}\quad v\in\XX.
\]
Throughout this paper, we will frequently make use of the fundamental relation
\begin{equation}\label{eq:Fenchel-identity}
v\in \partial \varphi (x) \,\,\Longleftrightarrow\,\, x\in \partial \varphi^* (v).
\end{equation}

A fundamental tool in the analysis of constrained optimization problems is the {\em indicator function} $\delta_\Omega(\cdot)$, associated with a nonempty closed convex set $\Omega\subset\XX$, which is defined to be $0$ when $x\in\Omega$ and $\infty$ otherwise. The subdifferential of the indicator function $\delta_\Omega$ at a point $\ox\in\Omega$ coincides with the {\em normal cone} to $\Omega$ at $\ox$, given by
\begin{equation}\label{def:Nor}
N_{\Omega}(\ox):=\big\{v\in \XX\mid \la v,x-\ox\ra\le 0 \quad\forall\, x\in \Omega\big\}.
\end{equation}

\section{Solution uniqueness and isolated calmness of regularized linear inverse problem}\label{sec:main}

In this section, we consider the following regularized linear inverse problem
\begin{equation}\label{p:P}
 P(b):\qquad    \min_{x\in \XX}\quad R(x)=h(Kx)\quad \mbox{subject to}\quad \Phi x=b,
\end{equation}
where $h:\mathbb{E}\to \R$ is a nonnegative continuous convex function, $K:\XX\to \mathbb{E}$ and $\Phi:\XX\to \YY$ are linear operators between Euclidean spaces, and $b\in \YY$. Suppose that $x_0$ is an optimal solution of $P(b_0)$ with $b_0:=\Phi x_0$. %Let us introduce the standard assumption for the function $h$ that is used throughout this paper.

%\noindent {\bf Assumption}:
%\begin{itemize}
  %  \item[\bf (A)] The function $h$ is continuous.$$ around $x_0$. %or  polyhedral in the sense that $\epi h$ is a convex polyhedral.
%    \item[\bf(A2)] $\Ker R_\infty\cap\Ker \Phi=\{0\}$, where $R_\infty$ is known as the {\em asymptotic/horizon function} defined by 
%\begin{equation}\label{AF}
 %   R_\infty(w):=\liminf_{w^\prime\to w, t\to \infty}\dfrac{R(tw^\prime)}{t}
%\end{equation}
%and $\Ker R_\infty:=\{w\in \R^n|\; R_\infty(w)=0\}$. 
%\end{itemize}

Define the {\em solution mapping} as
\begin{equation}\label{eq:sol_map}
    S(b):=  \argmin\,\{R(x)|\, \Phi (x)=b \}\quad \text{for}\quad b\in \YY.
\end{equation}
For each $b\in \YY$, note that the set $S(b)$ is closed and convex.
%For any $x_\delta \in F(b)$ and $x_\delta \to x$, since $\Phi$ is a linear operator, $x\in F(b)$. Hence, $F(b)$ is closed. On the other hand, we claim that $F(b)$ is bounded. Indeed, suppose that $F(b)$ is unbounded, there exists sequence $\{x_\delta\} \subset F(b)$ such that $\|x_\delta\| \to \infty$. Without loss of generality, suppose that $\frac{x_\delta}{\|x_\delta\|} \to w$ and $\|
%w\|=1$, which implies that $w \in \Ker \Phi$. Moreover,
%\[
%0=\liminf_{\delta\dn 0}\dfrac{R(x_0)}{\|x_\delta\|}\ge \liminf_{\delta\dn 0}\dfrac{R(x_\delta)}{\|x_\delta\|}\ge R_\infty(w)\ge 0, 
%\]
The set of {\em pre-dual certificates} at $x_0$ is defined by 
\begin{equation}\label{def:DC}
    \Lm(x_0):=\{z\in \mathbb{E}|\, z\in \partial h(Kx_0)\, \text{ and }\, K^*z\in \Im \Phi^*\}. 
\end{equation}

Optimality of problem \eqref{p:P} at $x_0\in \XX$ can be expressed via $\Lambda(x_0)$, as established next.

\begin{Proposition}[Source  condition]\label{prop:SC} A feasible point $x_0$ is an optimal solution of $P(b_0)$ if and only if $\Lm(x_0)\neq \emptyset$.   
\end{Proposition}
\begin{proof}
Under the global continuity of $h$, observe that 
\[
\partial (R+\delta_{\Phi^{-1}(b_0)})(x_0)= \partial R(x_0) + \partial \delta_{\Phi^{-1}(b_0)}(x_0) =\partial R(x_0)+ N_{\Phi^{-1}(b_0)}(x_0)= K^*\partial h(Kx_0)+\Im \Phi^*.
\]
Hence $x_0$ is an optimal solution of problem $P(b_0)$ if and only if $0\in \partial (R+\delta_{\Phi^{-1}(b_0)})(x_0)$, which means $\Lm(x_0)\neq \emptyset$.  
\end{proof}

% \begin{Remark}
% When $x_0$ is the unique solution of $P(b_0),$ {\rm \cite[Remark~3.2]{FNT23}} helps us to derive the following condition:
% \begin{equation}\label{AC}
%     \Ker R_\infty\cap \Ker \Phi=\{0\},
% \end{equation}
% \end{Remark}

Let us establish the first main result of the paper that provides a characterization for solution uniqueness of problem \eqref{p:P} in terms of initial data and explicit functions $h$ and $K$, which extends \cite[Theorem~4.1]{FNP25} to the case of composite functions.

\begin{Theorem}[Characterization of solution uniqueness] \label{thm:Uniq} Suppose that $x_0\in \XX$ is an optimal solution of problem $P(b_0)$. Then the following are equivalent: 
\begin{itemize}
    \item[{\bf(i)}] $x_0$ is the unique solution of $P(b_0)$.
    \item[{\bf(ii)}] For any $z\in \Lm(x_0)$, 
    \begin{equation}\label{con:Rad}
        \Ker \Phi \cap K^{-1}\left(\cone(\partial h^*(z)-Kx_0)\right)=\{0\}.
    \end{equation}
    \item[{\bf(iii)}] For any $z\in \Lm(x_0)$, 
    \begin{equation}\label{con:original}
        \Ker \Phi \cap K^{-1}\left(\partial h^*(z)-Kx_0\right)=\{0\}.
    \end{equation}
    \item[{\bf(iv)}] There exists some $z\in \Lm(x_0)$ such that condition~\eqref{con:Rad} is satisfied.
    \item[{\bf(v)}] There exists some $z\in \Lm(x_0)$ such that condition~\eqref{con:original} is satisfied.
\end{itemize}
\end{Theorem}
\begin{proof}
Since [{\bf(ii)}$\Rightarrow${\bf(iii)}$\Rightarrow${\bf (v)}] and [{\bf(ii)}$\Rightarrow${\bf(iv)}$\Rightarrow${\bf (v)}] are trivial, it suffices to verify [{\bf(i)}$\Rightarrow${\bf(ii)}] and [{\bf(v)}$\Rightarrow${\bf(i)}]. Let us start to prove [{\bf(i)}$\Rightarrow${\bf(ii)}] by supposing that $x_0$ is the unique solution of $P(b_0)$. Fix $z\in \Lambda (x_0)$. Pick any $w\in \Ker \Phi \cap K^{-1}\left(\cone(\partial h^*(z)-Kx_0)\right)$, which implies that 
$$Kw\in \cone(\partial h^*(z)-Kx_0).$$ 
If $Kw=0$, then by the optimality of $x_0$ to $P(b_0)$,
\begin{equation*}
    0\in K^* \partial h\big(Kx_0\big)+ \Im \Phi^* = K^* \partial h\big(K(x_0+w)\big)+ \Im \Phi^*.
\end{equation*}
This results in $\Lambda (x_0+w)\neq \varnothing$, and thus $$x_0+w \in S\big(\Phi (x_0+w)\big) = S(b_0) = \{x_0\},$$ which yields $w=0$. If $Kw\neq 0$, we find some $t>0$ such that $K(x_0+tw)\in \partial h^*(z)$. It follows that  $z\in \partial h(K(x_0+tw))$. As $K^*z\in \Im \Phi^*$, we obtain $z\in\Lm(x_0+tw)$. Since $\Phi(x_0+tw)=\Phi x_0=b_0$, $x_0+tw$ is also an optimal solution of $P(b_0)$ by Proposition~\ref{prop:SC}. This asserts that $w=0$. Hence [{\bf (i)}$\Rightarrow${\bf (ii)}] is verified.

It remains to justify  [{\bf(v)}$\Rightarrow${\bf(i)}]. Suppose that condition~\eqref{con:original} is satisfied at some $z\in \Lm(x_0)$. By contradiction, let $x_1$ be another optimal solution of $P(b_0)$ that is different from $x_0$ and set $w := x_1 - x_0 \ne 0$. Since both $x_0$ and $x_1$ are optimal solutions, we have $R(x_1)=R(x_0)$ and $\Phi x_1=\Phi x_0=b_0$. Thus, one gets $\Phi w = \Phi x_1 - \Phi x_0 = 0$, and therefore $w\in \Ker \Phi$. Moreover, because $z\in \Lm(x_0)$, we have $\la Kw, z\ra = \la w, K^*z\ra=0$ as $w\in \Ker \Phi$ and $K^*z\in \Im \Phi^*$. Therefore,
\begin{equation}
h\big(K(x_0+w)\big)=h(K x_1) = h(Kx_0)=\la Kx_0,z \ra-h^*(z)=\la K(x_0+w),z\ra-h^*(z),
\end{equation}
which implies that $K(x_0+w)\in \partial h^* (z)$ and 
\[
Kw = K(x_0 + w) - Kx_0 \in \partial h^*(z) - Kx_0.
\]
Hence, we derive that $w\in K^{-1}\left(\partial h^*(z)-Kx_0\right) \cap \Ker \Phi = \{ 0\}$, which is a contradiction. In conclusion, $x_0$ is the unique solution of $P(b_0)$.
\end{proof}

\begin{Remark} {\rm 
When $\mathbb{E} = \XX$ and $K=\mathbb{I}_{\XX}$ is the identity operator, the equivalence between {\bf (i)}, {\bf (ii)}, and {\bf (iv)} was obtained in \cite[Theorem~4.1]{FNP25}. The equivalence between  {\bf (i)}, {\bf (iii)}, and {\bf (v)} here looks simpler. But condition \eqref{con:Rad} can be more verifiable in practice, especially when the conic hull there is a closed set, which happens when $h$ is a {\em polyhedral function} (a.k.a. convex {\em piecewise linear} function); see \cite[Remark~3.1]{FNP25} for further discussions. In the next example, we illustrate how the geometry condition \eqref{con:Rad} is explicitly interpreted, in the case $h$ is the $\ell_1$-norm \eqref{defi:l1}. 
}
\end{Remark}

\begin{Example}[Solution uniqueness when $h=\|\cdot\|_1$]\label{re:l1}\rm When $\mathbb{E}=\R^m$, consider the $\ell_1$ norm 
\begin{equation}\label{defi:l1}
h(u)=\|u\|_1 :=\sum_{i=1}^m|u_i|
\end{equation}
and the composite function $R(x)=h(Kx)=\|Kx\|_1$, where $K$ is an $m\times n$ matrix. Denote the index set $I:=\supp(Kx_0) = \{i \in \{1,\ldots,m\} |\, (Kx_0)_i \ne 0 \}.$ Note that $h^*=\delta_{\mathbb{B}_\infty}$ with 
\[
\B_\infty:=\{z\in \R^m|\, \|z\|_\infty\le 1\}\quad \mbox{and}\quad \|z\|_\infty:=\max\{|z_i|\,|\, i=1,\ldots,m\}\quad \mbox{for}\quad z\in \R^m.
\]
Observe that 
$$
\begin{aligned}
\Lambda(x_0)
&= \{z\in \R^m |\, z\in \partial \|\cdot\|_1 (Kx_0),\, K^* z \in \Im \Phi^*\}\\
&=
\big\{
z\in \mathbb{R}^m \big|\,
z_I=\sign(Kx_0)_I,\ \|z_{I^c}\|_\infty\le 1,\ K^*z\in \Im \Phi^*
\big\}.
\end{aligned}
$$
Fix any $z\in \Lambda(x_0)$ and define $J_z:=\{i\in\{1,\ldots,m\}|\, |z_i|=1\},$
which implies that $I\subset J_z$. It follows that
\begin{equation*}
    Kx_0 \in \partial h^* (z) = \partial \delta_{\mathbb{B}_{\infty}}(z)=N_{\mathbb{B}_{\infty}}(z) = \big(\R_+ z_i\big)_{i\in J_z}\times \{0\}_{J_z^c}.
\end{equation*}
Hence, 
$$
\cone\big(\partial h^*(z)-Kx_0\big) = \cone\Big(\big[(\R_+ z_i)_{i\in J_z}\times \{0\}_{J_z^c}\big]-Kx_0\Big)
\subset \R^{I}\times \big(\R_+ z_i\big)_{i\in J_z\setminus I}\times \{0\}_{J_z^c}.
$$
The reverse inclusion also holds true. Indeed, pick any $v$ from the right-most set. Since $z_i=\sign(Kx_0)_i$ for any $i\in I$, we can find $t>0$ sufficiently small so that $$tv_i+(Kx_0)_i\in \R_+z_i \,\,\, \text{for every}\,\,\, i\in I.$$ 
% and find $t>0$ sufficiently small so that 
% \begin{equation*}
%     t|v_i| < \dfrac{|(Kx_0)_i|}{2} \quad \forall\, i\in I.
% \end{equation*}
% Under this choice, a direct computation shows that $tv_i + (Kx_0)_i \in \R_+ z_i$ for each $i\in I$. 
This results in the inclusion $v\in \cone\big([(\R_+ z_i)_{i\in J_z}\times \{0\}_{J_z^c}]-Kx_0\big)$, which yields 
\begin{equation}\label{cone:l1}
    \cone\big(\partial h^*(z)-Kx_0\big) =\R^{I}\times \big(\R_+ z_i\big)_{i\in J_z\setminus I}\times \{0\}_{J_z^c}.
\end{equation}
Hence, condition~\eqref{con:Rad} becomes
\[
\Ker \Phi \cap K^{-1}\left(\R^{I}\times \big(\R_+ z_i\big)_{i\in J_z\setminus I}\times \{0\}_{J_z^c}\right)
= \{0\},
\] which can be equivalently written as
\begin{equation}\label{cond:unique1}
\Ker \Phi \cap
\Big\{
w\in \R^n \Big|\,
Kw\in \R^{I}\times \big(\R_+ z_i\big)_{i\in J_z\setminus I}\times \{0\}_{J_z^c}
\Big\}
=
\{0\}.
\end{equation}
Therefore, solution uniqueness can be checked by searching for a pre-dual certificate $z\in\Lambda(x_0)$
for which the above condition holds. In Remark~\ref{re:calm-l1}, we show that this condition is explicitly verifiable via solving Linear Programming (LP) optimization problems.
\end{Example}

Let us recall the central notion of this paper, the isolated calmness property; see, e.g., \cite[Section~3.9]{DR14}.
\begin{Definition}[Isolated calmness property]\label{defi:IC}
Let $x_0$ be an optimal solution of $P(b_0)$.
The solution mapping $S(\cdot)$ in \eqref{eq:sol_map} is said to satisfy the {\em isolated calmness property} at $b_0$ for $x_0$ if the assertion
\begin{equation}\label{defi:iso-calm}
    S(b)\cap U \subset \{x_0\} + \kappa \|b-b_0\| \mathbb{B}_{\XX} \qquad \forall\, b\in \mathbb{B}_{\delta}(b_0)
\end{equation}
holds for some neighborhood $U$ of $x_0$ and some parameters $\kappa,\delta>0$. 
\end{Definition}

In the presence of the isolated calmness property of $S$ at $b_0$ for $x_0$, the solution set $S(b_0)$ reduces to the singleton $x_0$ due to its convexity; that is, $x_0$ is the unique solution of problem \eqref{p:P}. Moreover, the restriction to the neighborhood $U$ in \eqref{defi:iso-calm} can be removed due to the {\em inner semi-continuity} \cite[Section~3.2]{DR14} of the solution mapping, which will be established next.

\begin{Proposition}[Inner semi-continuity of the solution mapping $S(\cdot)$]\label{converge}
Let the linear operator $\Phi:\XX\to\YY$ have full row rank. Assume that $x_0$ is the unique solution to $P(b_0)$, and let $\{\delta_k\}$ be a sequence satisfying $\delta_k\to 0$. For each $k\in\mathbb{N}$, choose $b_k\in\mathbb{B}_{\delta_k}(b_0)$ and suppose that $x_k\in S(b_k)$. Then the sequence $\{x_k\}$ converges to $x_0$ as $k\to\infty$.

Consequently, $S$ is {\em inner semi-continuous} at $b_0$ (for $x_0$) in the sense that for every neighborhood $U$ of $x_0$, there exists $\delta>0$ such that for each $b\in \mathbb{B}_{\delta}(b_0)$ and $x\in S(b)$, we have $x\in U$.
\end{Proposition}
\begin{proof}
We consider the (Fenchel) dual problem of $P(\cdot)$ via the function $v: \YY \to \oR$ as follows:
\begin{equation}\label{p:pri}
    v(b):=\,\sup_{y\in \YY}\,\{\la y,b\ra- R^*(\Phi^*y)\}\qquad (b\in \YY).
\end{equation}
By Fenchel's Duality Theorem \cite[Corollary~31.2.1]{R70}, it follows that
\begin{equation}\label{DP}  
    \inf_{\{x\,|\,\Phi x=b\}} \, R(x) = \max_{y\in \YY}\,\{\la y,b\ra- R^*(\Phi^*y)\} = v(b) \qquad \forall\, b\in \Phi\big(\mathrm{ri}\, (\dom R)\big).
\end{equation}
Actually, we have $\Phi\big(\mathrm{ri}\, (\dom R)\big) = \YY$, since $\Phi$ is full row rank and $R(\cdot)$ is a continuous function. Consequently, $v(\cdot)$ is convex and has full domain in $\YY$. In particular, it is locally Lipschitz around $b_0$ with some modulus $L_1>0$. For each $k\in \N$, the relationship in \eqref{DP} lends us some $y_k\in \YY$ such that 
\begin{equation}\label{eq:vbk}
    v(b_k)=\la y_k,b_k\ra- R^*(\Phi^*y_k).
\end{equation}
It then follows from the max rule of convex subdifferential calculus that $y_k\in \partial v(b_k)$. Consequently, we have $\|y_k\|\le L_1$ for large enough $k\in \N$. As $x_k$ solves $P(b_k)$, it follows from \eqref{DP}--\eqref{eq:vbk} that
\begin{eqnarray}\label{eq:bound-Rxk}
\begin{aligned}
            R(x_k)=v(b_k)=\la y_k,b_k\ra- R^*(\Phi^*y_k)&=\la y_k,b_k-b_0\ra+ \la y_k,b_0 \ra - R^*(\Phi^*y_k) \\
            &\le \|y_k\|\|b_k-b_0\|+v(b_0)\\
            &\le L_1\delta_k+R(x_0).
\end{aligned}
\end{eqnarray}
% \textcolor{blue}{2. Primal:} Since $\Phi$ is full row rank, it admits a linear right inverse $\Psi$, i.e., $\Phi \circ\Psi = I_{\YY}$. Define $h_k := \Psi (b_k-b_0)$, we have $h_k \to 0$ as $k\to \infty$ and 
% \begin{equation}\label{eq:Rxkle}
%     \Phi (x_0 + h_k) = \Phi (x_0) + (\Phi \circ \Psi)(b_k-b_0) = b_0 + b_k - b_0 = b_k.
% \end{equation}
% As $x_k$ solves $P(b_k)$, it follows that $$R (x_k) \le R (x_0+h_k).$$
% Since $h_k\to 0$, we consider $k\in \N$ large enough so that $x_0 + h_k$ belongs to a neighborhood of $x_0$ on which $R$ is locally Lipschitz with some modulus $L_0>0$. Consequently, by \eqref{eq:Rxkle},
% \begin{equation*}
%     R(x_k) \le R(x_0+h_k) \le R(x_0) + L_0 \|h_k\| \le R(x_0) + L_0 \|\Psi\|\|b_k-b_0\|\le R(x_0) + L_0 \|\Psi\| \delta_k.
% \end{equation*}
% \\
We claim that the sequence $\{x_k\}$ is bounded when $\delta_k \to 0$. By contradiction suppose that $\{x_k\}$ is unbounded, and without loss of generality assume that $\|x_k\|\to \infty$ and $w_k:=\frac{x_k}{\|x_k\|}\to w$ with $\|w\|=1$. Moreover, since $\Phi (x_k) = b_k$, by dividing both sides by $\|x_k\|$ and passing to the limits as $\delta_k\to 0$, we have $w\in \Ker \Phi$.

Recall from \cite{AT03} the {\em asymptotic function} $R_\infty:\XX \to \oR$ associated with $R$, defined by
\begin{equation}\label{AF}
    R_\infty(w):=\liminf_{w^\prime\to w,\, t\to \infty}\dfrac{R(tw^\prime)}{t},
\end{equation}
and let $\Ker R_\infty:=\{w\in \XX|\; R_\infty(w)=0\}$. The relations in \eqref{eq:bound-Rxk} and \eqref{AF} and the nonnegativity of $R(\cdot)$ together deduce the estimates
\[
0=\liminf_{k\to \infty}\dfrac{L_1 \delta_k + R(x_0)}{\|x_k\|}\ge \liminf_{k\to \infty}\dfrac{R(x_k)}{\|x_k\|}\ge R_{\infty}(w)\ge 0, 
\]
and thus one obtains $w\in \Ker R_\infty$. Hence $w\in \Ker \Phi \cap \Ker R_{\infty}$ and $\|w\|=1$. This contradicts \cite[Remark~3.2]{FNT23} as $x_0$ is the unique solution of $P(b_0)$. Thus $\{x_k\}$ is bounded,  and there exists a subsequence $\{x_{k_m}\}$ of $\{x_k\}$ that converges to some vector $u_0$ as $m\to \infty$. By passing to the limits, we have $\Phi u_0=b_0$. Moreover, optimality of the problems $P(b_{k_m})$ $(m\in \N)$ lends us a sequence $\{v_{k_m}\}$ satisfying 
\begin{equation}\label{eq:vkm}
    v_{k_m}\in \partial R(x_{k_m})\cap {\rm Im}\, \Phi^*.
\end{equation}
Suppose that $R(\cdot)$ is locally Lipschitz around $u_0$ with some modulus $L_2>0$. This implies that $\|v_{k_m}\|\le L_2$. Hence, without passing to subsequences suppose that $v_{k_m} \to v$. Consequently, we get from \eqref{eq:vkm} the inclusion 
$$
    v\in \partial R(u_0)\cap {\rm Im}\,\Phi^*.
$$ 
Thus, $u_0$ is an optimal solution to $P(b_0)$, and by uniqueness, there holds $u_0=x_0$. Since the subsequence $\{x_{k_m}\}$ was taken arbitrarily, we eventually have $x_k \to x_0$ as $k\to \infty$. 

The second assertion follows immediately from the first by contradiction.
\end{proof}

\begin{Remark}\label{rem:iso-calm}
\rm 
%\begin{enumerate}
By Proposition~\ref{converge}, the neighborhood $U$ in the definition of isolated calmness \eqref{defi:iso-calm} can be omitted. More precisely, $S$ has the isolated calmness at $b_0$ for $x_0$ if and only if there exist $\kappa,\delta>0$ for which the assertion
    \begin{equation}\label{charac:iso-calm}
        S(b) \subset \{x_0\} + \kappa \|b-b_0\|\mathbb{B}_{\XX}\qquad \forall\, b\in \mathbb{B}_{\delta}(b_0)
    \end{equation}
    holds.
    %\item[{\bf (ii)}] It is immediate to show that $S$ has the isolated calmness property \eqref{charac:iso-calm} at $b_0$ for $x_0\in S(b_0)$ if and only if there exist positive constants $K,\delta_0>0$ such that for every $\delta\in (0,\delta_0)$ and $b\in  \mathbb{B}_{\delta}(b_0)$, any solution $x_b\in S(b)$ satisfies \begin{equation*}
    %\|x_b-x_0\|\le K\delta.
%\end{equation*} 
%This is also known as the ``{\em linear rate of convergence}'', which has some roots in the definition of {\em stable recovery}; see, e.g., \cite[Definition~3.1]{NPV25}.
%\end{enumerate}
\end{Remark}

Next, we establish the main result of this paper, that provides a full characterization of isolated calmness for the solution mapping $S(\cdot)$ via the tangent cone. 

\begin{Theorem}[Geometric characterization of isolated calmness]\label{thm1} Suppose that $\Phi$ is full row rank and that $x_0$ is an optimal solution of $P(b_0)$. Then the solution mapping $S(\cdot)$ has the isolated calmness property at $b_0$ for $x_0$ if and only if
\begin{equation}\label{con:isolated-calm}
    \Ker \Phi \cap T_{\partial R^*({\rm Im}\,\Phi^*)}(x_0)=\{0\}.
\end{equation}
\end{Theorem}
\begin{proof}
Assume that isolated calmness of $S(\cdot)$ occurs at $b_0$ for $x_0$ with respect to positive constants $\kappa,\delta>0$. Take any vector $w\in \Ker \Phi \cap T_{\partial R^*({\rm Im}\,\Phi^*)}(x_0).$ By the definition of the tangent cone \eqref{eq:tan}, there exist sequences $t_k\downarrow 0$ and $w_k \to w$ such that 
$$
    x_k:=x_0+t_kw_k\in \partial R^*({\rm Im}\, \Phi^*).
$$ 
Hence, for each $k$, there exists some $v_k\in {\rm Im}\,\Phi^*$ such that $v_k\in \partial R(x_k)$. Let us define
        \begin{equation}\label{def:del,b}
        \delta_k:=\|\Phi(x_k-x_0)\|=t_k\|\Phi w_k\|\to 0 \quad \mbox{and} \quad b_k:=\Phi x_k \to b_0.
        \end{equation}
        It follows that
        \begin{equation}\label{con:conb}
            \|b_k-b_0\|=\|\Phi(x_k-x_0)\|=\delta_k \to 0 \text{ as }k\to \infty.
        \end{equation}
        Furthermore, since $v_k\in  \partial R(x_k)\cap {\rm Im}\, \Phi^*$ we have $$0\in \partial R(x_k) + {\rm Im}\, \Phi^*= \partial \big(R+\delta_{\Phi^{-1}(b_k)}\big)(x_k).$$ Combining with $\Phi x_k=b_k$, it follows that $x_k$ is an optimal solution of $P(b_k)$. The convergence in \eqref{con:conb} and the isolated calmness of $S(\cdot)$ at $b_0$ for $x_0$ (see Definition~\ref{defi:IC}) together yield
\[t_k\|w_k\|=\|x_k-x_0\|\le \kappa\delta_k = t_k \kappa\|\Phi w_k\|\]
for every $k\in \N$ large enough. This relationship leads to $\|w_k\|\le \kappa\|\Phi w_k\|$. As $w_k\to w \in \Ker \Phi$, we get from the latter that $\|w\|\le 0$, i.e., $w=0$, which verifies condition \eqref{con:isolated-calm}.\\

Conversely, let the condition \eqref{con:isolated-calm} hold. We suppose, toward a contradiction, that $S(\cdot)$ does not have the isolated calmness property at $b_0$ for $x_0$. Hence, there exist some sequences $\delta_k \downarrow 0$, $b_k \to b_0$ and $\{x_k\}\subset \XX$ such that 
\begin{equation}\label{eq:bkxk}
    \|b_k-b_0\|\le \delta_k \quad \text{and}\quad \|x_k-x_0\|\ge k\delta_k,
\end{equation}
where $x_k$ is an optimal solution of $P(b_k)$ for each $k\in \N$. Since $x_0$ solves $P(b_0)$, one finds some $v\in \partial R(x_0)\cap {\rm Im}\,\Phi^*$. This yields $\partial R^*(v) \subset \partial R^*( {\rm Im}\, \Phi^*)$, which results in
\[\{0\}=\Ker \Phi \cap T_{\partial R^*({\rm Im}\,\Phi^*)}(x_0) \supset \Ker \Phi \cap\left(\cone(\partial R^*(v)-x_0)\right).\] 
It follows from \cite[Theorem~4.1]{FNP25} that $x_0$ is the unique solution of $P(b_0)$. As $x_k$ solves $P(b_k)$, Proposition~\ref{converge} tells us that $\{x_k\}$ converges to $x_0$ as $k\to \infty$.
We define $t_k:=\|x_k-x_0\|\ge k\delta_k >0$ and $q_k:=t_k^{-1}(x_k-x_0)$. By taking subsequences if needed, suppose without loss of generality that $t_k\dn 0$ and $q_k\to q\in \XX$ with  $\|q\|=1$. Pick $v_k\in \partial R(x_k)\cap \Im \Phi^*$. The relationships 
\begin{equation*}
    q_k t_k + x_0 = x_k \in \partial R^* (v_k) \subset \partial R^* (\mathrm{Im}\, \Phi^*)
\end{equation*}
clearly yield $q\in T_{\partial R^* (\mathrm{Im}\, \Phi^*)}(x_0)$. Moreover, it follows from \eqref{eq:bkxk} that
\[
\|\Phi^*(\Phi q_k)\|\le\|\Phi^*\|\dfrac{\|\Phi(x_k-x_0)\|}{t_k}=\|\Phi^*\|\dfrac{\|b_k-b_0\|}{t_k}\le\dfrac{\|\Phi^*\|}{k}\to 0 \ \text{ as }\ k\to \infty,
\]
which yields $\Phi^*(\Phi q)=0$, i.e., $\Phi q=0$. Hence, we have 
\[
q\in \Ker \Phi\cap T_{\partial R^*({\rm Im}\,\, \Phi^*)}(x_0). 
\]
This contradicts  \eqref{con:isolated-calm}, as $\|q\|=1$. The proof is complete.
\end{proof}

\begin{Remark}\rm Condition~\eqref{con:isolated-calm} was discovered in \cite[Theorem~3.3]{NPV25} to characterize the so-called {\em stable recovery} via {\em Tikhonov regularization problem}~\eqref{p:Lass} in the sense  of \eqref{eq:SC}. This somewhat proves that the stable recovery at $x_0$ of the linear inverse problem~\eqref{p:P} is equivalent to the isolated calmness of the optimal solution mapping $S$ in \eqref{eq:sol_map} at $b_0$ for $x_0$. Unlike \cite[Theorem~3.3]{NPV25}, we need to suppose additionally that $\Phi$ has full row rank, which simply avoids the case that the optimal set $S(b)$ is empty at some $b$ around $b_0$.
\end{Remark}

\begin{Remark}[Verification of isolated calmness in continuous convex piecewise linear-quadratic cases]\label{rem:calm-cpwl}
\rm  Let us work under the assumptions of Theorem~\ref{thm1}, and additionally assume that $R(\cdot)$ is a continuous convex and {\em piecewise linear--quadratic function} \cite[Definition~10.20]{Rockafellar98}, i.e., $\dom R$ is the
union of finitely many polyhedral sets and, on each of these sets, \(R\)
coincides with a quadratic function of the form
\[
x\mapsto \frac{1}{2}\langle Ax,x\rangle+\langle b,x\rangle+c,
\]
for some positive semidefinite operator from $\XX$ to itself, vector \(b\in\mathbb X\), and
scalar \(c\in\mathbb R\).

By Theorem~\ref{thm1}, and the arguments in the proof of \cite[Corollary~3.8]{NPV25}, the isolated calmness property of the solution mapping $S(\cdot)$ at $b_0$ for $x_0$ is equivalent to the uniqueness of the solution to $P(b_0)$ at $x_0$. Invoking Theorem~\ref{thm:Uniq}, this condition can be proved or disproved by simply taking any $z\in \Lambda(x_0)$ and checking  \eqref{con:Rad}. 
Clearly, verifying condition~\eqref{con:Rad} is more straightforward than checking~\eqref{con:isolated-calm}; see \cite{NPV25} for further discussions on checking~\eqref{con:isolated-calm} in different cases. Further, in Remark~\ref{re:calm-l1} below, we provide a concrete illustration of how the isolated calmness property can be justified via \eqref{con:Rad} in the case $R(x):=\|Kx\|_1$.
\end{Remark}

\begin{Remark}[Verification of isolated calmness property in the case $R(x)=\|Kx\|_1$]
\label{re:calm-l1}\rm Consider problem \eqref{p:P}, in which $(\XX,\mathbb{E}):= (\R^n,\R^m)$ and $h:=\|\cdot\|_1$ is the $\ell_1$-norm \eqref{defi:l1} on $\R^m$. This composite function $R$ is a piecewise linear function. Fix $x_0\in \R^n$ at which the source condition \begin{equation*}
    \emptyset \neq \Lambda (x_0) := \{z\in \R^m|\, z\in \partial\|\cdot\|_1 (Kx_0) \text{ and } K^* z \in \Im \Phi^*\}
\end{equation*}
holds, i.e., $x_0$ yields the feasibility of the LP problem:
\begin{equation}\label{LP_z}
\begin{aligned}
    \min_{z,u} \quad &0 \\
    \text{s.t.}\quad &\Phi^*u = K^*z \\
    &z_I=\sign(Kx_0)_I \\ 
    &\mathbf{-1} \le z_{I^c} \le \mathbf{1}.
\end{aligned}
\end{equation}
We pick any pre-dual certificate $z\in \Lambda (x_0)$, which can be done by solving \eqref{LP_z}.
Subsequently, by Remark~\ref{rem:calm-cpwl} and equation \eqref{cond:unique1}, verifying the isolated calmness property of $S(\cdot)$ at $b_0 = \Phi (x_0)$ for $x_0$ reduces to justifying the identity 
\begin{equation*}
    \Ker \Phi \cap
    \Big\{
    w\in \R^n \Big|\,
    Kw\in \R^{I}\times \big(\R_+ z_i\big)_{i\in J_z\setminus I}\times \{0\}_{J_z^c}
    \Big\}
    =
    \{0\},
\end{equation*}
in which $I:=\{i|\, (Kx_0)_i \neq 0\}$ and $J_z := \{i|\, |z_i|=1\}$. This amounts to checking whether
\begin{equation*}
\not\exists\, w\neq 0 \text{ such that }
\begin{cases}
    \Phi w = 0 \\
    K_{J_z^c} w = 0 \\
    (z_i K_i)_{i\in J_z \setminus I} w \ge 0,
\end{cases}
\end{equation*}
or equivalently in matrix form,
\begin{equation}\label{check:matrix}
\not\exists\, w\neq 0 \text{ such that }
\begin{cases}
    Aw=0\\
    Bw\ge 0 ,
\end{cases}
\end{equation}
where $A:=\begin{pmatrix}
    \Phi \\
    K_{J_z^c}
\end{pmatrix}$, $B:= (z_i K_i)_{i\in J_z \setminus I}$, and $K_i$ stands for the $i$-th row of $K$. 

When $J_z \setminus I = \emptyset$, verifying \eqref{check:matrix} is equivalent to checking $\Ker A = \{0\}$. We show next that when $J_z \setminus I \ne \emptyset$, one can verify \eqref{check:matrix} via linear programming. By employing Gordan's Lemma, we show that
\begin{equation}\label{check:split}\eqref{check:matrix}\quad  \Longleftrightarrow \quad  \Ker \begin{pmatrix}
        A \\ B
    \end{pmatrix} = \{0\} \,\text{ and }\, \exists\, \mu>0,\exists\, \lambda : A^* \lambda + B^* \mu =0. 
\end{equation}
Indeed, we verify ``$\Longleftarrow$" by supposing, toward a contradiction, that $\Ker \begin{pmatrix}
        A \\ B
\end{pmatrix} = \{0\} \text{ and } \exists\, \mu>0,\exists \lambda : A^* \lambda + B^* \mu =0$, and there exists $w\neq 0$ satisfying $Aw=0,Bw\ge 0$. Then 
\begin{equation*}
\begin{aligned}
    0 = \la A^*\lambda + B^* \mu, w \ra = \la \lambda , Aw\ra + \la \mu, Bw\ra = \la \mu, Bw\ra \ge 0,
\end{aligned}
\end{equation*}
which yields $Bw=0$. Having $Aw=Bw=0$ and $w\neq 0$ clearly contradict $\Ker \begin{pmatrix}
        A \\ B
\end{pmatrix} = \{0\}$. Thus the implication ``$\Longleftarrow$" is verified.

Conversely, suppose \eqref{check:matrix} is true. Clearly, we have $\Ker \begin{pmatrix}
    A \\ B
\end{pmatrix} = \{0\}$ in this case. Let $N$ be a matrix whose columns form a basis of $\Ker A$. Then every solution of $Aw=0$ can be written uniquely as $w=Nu$ for some vector $u$. Setting $C:=BN$, we have the relation
\begin{equation*}
     Aw=0,Bw\ge 0\,\, \Longleftrightarrow w=Nu \,\, \text{and} \,\, C u\ge 0.
\end{equation*}
Further, because the columns of $N$ are linearly independent, we have $w=0$ if and only if $u=0$. Consequently, \eqref{check:matrix} implies that $\not\exists\, u\neq 0$ for which $Cu\ge 0$. In other words,
\begin{equation*}
    \Im C \cap \R^p_+ = \{0\},\quad \text{ where }\, p:=|J_z \setminus I|\, \text{ is the cardinality of }\, J_z \setminus I.
\end{equation*}
Let $D$ be a matrix such that $\Im D = \Ker C^*$, or equivalently, $\Ker D^* = \Im C$. We have the relations
\begin{equation*}
\begin{aligned}
    \Im C \cap \R^p_+ = \{0\} \Longrightarrow \Ker D^* \cap  \R^p_+ = \{0\} &\Longrightarrow \not\exists\, q\ge 0,q\neq 0: D^* q =0 \\
    &\Longrightarrow \exists\, z: Dz>0 \\
    &\Longrightarrow \Ker C^* \cap \R^p_{++} = \Im D \cap \R^p_{++} \neq \varnothing \\
    &\Longrightarrow \exists\, \mu >0: N^* B^* \mu = C^* \mu =0,
\end{aligned}
\end{equation*}
where the third one is due to Gordan's Lemma \cite{Gordan1873}. This lends us some $\mu>0$ satisfying 
\begin{equation*}
    B^* \mu \in \Ker N^* = (\Im N)^\perp = (\Ker A)^\perp = \Im A^*,
\end{equation*}
which eventually verifies the implication ``$\Longrightarrow$", and further verifies the equivalence in \eqref{check:split}.\\

Based on the above analysis, verifying the isolated calmness property can now be reformulated into justifying the following two conditions, namely 
\[
\Ker \begin{pmatrix}
    A\\ B
\end{pmatrix} = \{0\}
\quad \text{and} \quad \exists\, \mu >0,\exists\,\lambda: A^* \lambda + B^* \mu =0.
\]
% $\Ker \begin{pmatrix}
%     A\\ B
% \end{pmatrix} = \{0\}$ and 
% \begin{equation}\label{check:A*B*}
%     \exists \mu >0,\exists\lambda: A^* \lambda + B^* \mu =0.
% \end{equation}
The former can be checked by ensuring the number of positive singular values is the same as number of columns of the matrix  $\begin{pmatrix}
    A\\ B
\end{pmatrix}
$ via the singular value decomposition (SVD). Following this, justifying the latter could be done by checking feasibility of solving the following LP problem:
\begin{equation}\label{LP-test}
\begin{aligned}
    \min_{\lambda,\mu} \quad &\mu_1 \\
    \text{s.t.}\quad &A^*\lambda + B^* \mu = 0 \\
    &\mu \ge \mathbf{1}.
\end{aligned}
\end{equation}
\end{Remark}

Next, we examine how condition \eqref{con:Rad} can be reformulated for an important subclass of convex piecewise linear-quadratic functions, namely the class of \emph{piecewise linear-quadratic penalties} \cite[Example~11.18]{Rockafellar98}.

\begin{Definition}[Piecewise linear-quadratic penalty]
Suppose that $B\in \mathbb R^{m\times m}$ is a symmetric positive semidefinite matrix and $\mathcal P\subseteq \mathbb R^m$ is a nonempty polyhedral set. The piecewise linear-quadratic penalty generated by $\mathcal P$ and $B$, denoted by $\theta_{\mathcal P,B}$, is given by
\begin{equation}\label{eq:theta}
    \theta_{\mathcal P,B}(y):=\sup_{z\in\mathcal P}\left\{\la y,z \ra -\frac12 
    \la Bz,z\ra \right\},\qquad y\in \R^m.
\end{equation}
\end{Definition}

This function is proper, lower semicontinuous, and convex. Its Fenchel conjugate takes the form
\begin{equation}\label{eq:theta*}
    \theta_{\mathcal P,B}^*(z)=\delta_{\mathcal P}(z)+\frac12 \la 
    Bz,z\ra\quad \mbox{for every} \quad z\in\mathbb R^m.
\end{equation}
Suppose that the polyhedral set $\mathcal{P}$ is defined by 
\[
\mathcal{P}=\big\{y\in \R^m|\, \langle a_i,y\rangle\le c_i,\, i=1,\ldots,p\big\}
\]
with some $(a_i,c_i)\in \R^m \times \R$, $i=1,\ldots, p$ and $I(z):= \{i\in \{1,\ldots,p\}|\, \la a_i,z\rangle = c_i\rangle\}$. The following result provides an explicit form  for \eqref{con:Rad}.

\begin{Corollary}[Solution uniqueness and isolated calmness property when $h=\theta_{\mathcal{P},B}$]\rm 
In the case $h=\theta_{\mathcal{P},B}$ and $K$ is an $m\times n$ matrix, condition \eqref{con:Rad} becomes
\begin{equation}\label{cond:plqp}
    \Ker \Phi \cap K^{-1}\Big(\cone \big\{a_i|\, i\in I(z)\big\} + \R_+\big(Bz - Kx_0\big)\Big) = \{0\}\quad \mbox{for any }\quad z\in \Lm(x_0) .
\end{equation}
Under the constraint qualification condition

\begin{equation}\label{eq:cont_qual}
    \mathcal{P}^{\infty} \cap \Ker B = \{0\},
\end{equation}
the function $h$ is continuous. In this case, $S$ has the isolated calmness property at $b_0$ for $x_0$ if and only if \eqref{cond:plqp} holds for some, equivalently for every, $z\in \Lambda(x_0)$. 
\end{Corollary}
\begin{proof}

Let us start to prove that 
condition~\eqref{cond:plqp} is equivalent to
\eqref{con:Rad}. Indeed, using \eqref{eq:theta*}, we obtain
\begin{equation}\label{eq:parh*}
    \partial h^* (z) = N_{\mathcal{P}}(z) + Bz,\qquad \forall\, z\in \R^m. 
\end{equation}
Suppose that $x_0$ solves $P(b_0)$. Pick some pre-dual certificate $z$ satisfying $z\in \partial h(Kx_0)$ and $K^* z \in \Im \Phi^*$. By the relations in \eqref{eq:Fenchel-identity} and \eqref{eq:parh*}, we have $Kx_0 \in \partial h^* (z) = N_{\mathcal{P}}(z) + Bz$, which yields 
\begin{equation}\label{eq:absorp}
Kx_0 - Bz \in N_{\mathcal{P}}(z).
\end{equation}
Since $\partial h^* (z) - Kx_0 = N_{\mathcal{P}}(z) + (Bz - Kx_0)$, it follows that 
\begin{equation*}
    \cone \big(\partial h^* (z) - Kx_0\big) = \cone \big(N_{\mathcal{P}}(z) + (Bz - Kx_0)\big) \subset N_{\mathcal{P}}(z) + \R_+(Bz - Kx_0).
\end{equation*}
The converse inclusion also holds true. Indeed, take $w = a + t(Bz-Kx_0)$ where $a\in N_{\mathcal{P}}(z)$ and $t\ge 0$. If $t>0$, then 
\begin{equation*}
    \dfrac{w}{t} = \dfrac{a}{t} + (Bz - Kx_0) \in N_{\mathcal{P}}(z) + (Bz-Kx_0). 
\end{equation*}
This leads to $w\in \cone \big(N_{\mathcal{P}}(z) + (Bz - Kx_0)\big)$. Otherwise, $t=0$ and \eqref{eq:absorp} implies that 
\begin{equation*}
    w = a = \big[a + (Kx_0-Bz)\big] + (Bz-Kx_0) \in N_{\mathcal{P}}(z) + (Bz-Kx_0), 
\end{equation*}
which also yields $w\in \cone \big(N_{\mathcal{P}}(z) + (Bz - Kx_0)\big)$. Therefore,
\begin{equation*}
    \cone \big(\partial h^* (z) - Kx_0\big) = N_{\mathcal{P}}(z) + \R_+(Bz - Kx_0) = \cone \{a_i|\, i\in I(z)\} + \R_+(Bz - Kx_0),
\end{equation*}
given that $\mathcal{P}=\big\{y\in \R^m|\, \langle a_i,y\rangle\le c_i,\, i=1,\ldots,p\big\}$ and $I(z):= \{i\in \{1,\ldots,p\}|\, \la a_i,z\rangle = c_i\rangle\}$. Condition \eqref{con:Rad} then becomes
\begin{equation*}
    \Ker \Phi \cap K^{-1}\Big(\cone \big\{a_i|\, i\in I(z)\big\} + \R_+\big(Bz - Kx_0\big)\Big) = \{0\},
\end{equation*}
as claimed. 

Next we show that under condition \eqref{eq:cont_qual}, the function $h$ is continuous over the whole space. Indeed, by \cite[Example~11.18]{Rockafellar98}, we have $\dom h = \dom \theta_{\mathcal{P},B}= (\mathcal{P}^{\infty} \cap \Ker B)^\circ$. Condition \eqref{eq:cont_qual} leads us to
\[
\dom h = (\mathcal{P}^{\infty} \cap \Ker B)^\circ= \R^n,
\]
which implies its continuity everywhere due to the convexity of $h$. By Remark~\ref{rem:calm-cpwl}, the solution
mapping \(S\) has the isolated calmness property at \(b_0\) for \(x_0\)
if and only if condition~\eqref{con:Rad} holds. The proof is complete.
\end{proof}

\begin{Remark}
    \rm  Condition~\eqref{eq:cont_qual} holds trivially when either  $\mathcal{P}$ is a {\em compact} polyhedral set or $B$ is an invertible matrix. In the recent paper,  \cite[Lemma~5.1]{SHN26} investigates the similar composite functions for a distinct purpose and use a weaker condition
    \[
    \mathcal{P}^{\infty} \cap \Ker B\cap \Ker K^*=\{0\},
    \]
which does not necessarily  imply the continuity of the function $h$.
    
\end{Remark}

Although the isolated calmness property of $S(\cdot)$ at $b_0$ for $x_0$ yields the uniqueness of $S(b_0)=\{x_0\}$, as noted in Remark~\ref{rem:iso-calm}, the converse implication does not hold in general. In fact, isolated calmness property can fail even when the solution mapping S is locally single-valued around the reference point. The following example, taken from \cite[Example~3.4]{NPV25}, demonstrates this phenomenon.

\begin{Example}[]\label{ex:NoSR}\rm
Consider the following $\ell_1/\ell_2$ optimization problem with two groups:
\begin{equation}\label{p:E12}
\min_{x\in \mathbb{R}^4} R(x) = \sqrt{x_1^2+x_2^2}+ \sqrt{x_3^2+x_4^2}, \quad \text{subject to} \quad \Phi x = b_0   
\end{equation}
with $\Phi=\begin{pmatrix} 1&0&0&-1\\0&1&0&1\\0&0&1&0\end{pmatrix}$,  $x_0=(0,1,0,0)^T$, and $b_0=(0,1,0)^T$.  Note that $$\Ker \Phi=\R(1,-1,0,1)^T \,\,\text{ and }\,\,\Im\Phi^*=\big\{(u_1,u_2,u_3,u_2-u_1)\in \R^4\big|\, (u_1,u_2,u_3)\in \R^3 \big\}.$$ Any feasible point $x$ takes the form $ x_0 + t(1,-1,0,1)^T$ for $t \in \mathbb{R}$.
Next, we claim that $\R_+(1,-1,0,1)^T\subset\Ker \Phi\cap T_{\partial R^*({\rm Im}\, \Phi^*)}(x_0)$. Indeed, fix $w:= (a,-a,0,a)^T$ for some $a>0$ and consider a sequence $t_k\dn 0$, we define the sequence $w_k:=(a,-a,c_k,d_k)^T$ with
\[
c_k:=\dfrac{a\sqrt{2t_ka-2t_k^2a^2}}{\|(t_ka,1-t_ka)\|} \quad\mbox{and} \quad d_k:=\dfrac{a(1-2t_ka)}{\|(t_ka,1-t_ka)\|}.
\]
Observe that $w_k \to w$ and $\|(c_k,d_k)\|=a\neq0$. Moreover,  we have
\[
v_k:=\nabla R(x_0+t_kw_k)= \left(\dfrac{t_ka}{\|(t_ka,1-t_ka)\|},\dfrac{1-t_ka}{\|(t_ka,1-t_ka)\|},\dfrac{\sqrt{2t_ka-2t_k^2a^2}}{\|(t_ka,1-t_ka)\|},  \dfrac{1-2t_ka}{\|(t_ka,1-t_ka)\|}\right)^T.
\]
It follows that $x_0+t_kw_k\in \partial R^*(v_k)$. Note also that  $v_k\in \Im \Phi^*$. Hence, $x_0+t_kw_k\in \partial R^*(\Im \Phi^*)$, which helps us to conclude that $w\in T_{\partial R^*({\rm Im}\,\, \Phi^*)}(x_0)$. By Theorem~\ref{thm1}, isolated calmness property of $S(\cdot)$ fails at $b_0$ for $x_0$.

Next, we show that $S(\cdot)$ is single-valued around $b_0$. Since the matrix $\Phi$ is full row rank, we have $
\operatorname{Im}\Phi=\mathbb{R}^3$. Let $b=(b_1,b_2,b_3)^T\in \operatorname{Im}\Phi$. Any $x=(x_1,x_2,x_3,x_4)^T \in \R^4$ satisfying $\Phi x=b$ yields 
$$
x_1-x_4=b_1,\, x_2+x_4=b_2,\,\text{ and }\, x_3=b_3.
$$
Hence every feasible solution of $P(b)$ can be written uniquely in the form
\[
x_b(t):=(b_1+t,\; b_2-t,\; b_3,\; t)^T,\qquad t\in\mathbb{R}.
\]
Therefore, solving $P(b)$ is equivalent to minimizing the univariate function
\begin{equation}\label{phi:12}
    \varphi_b(t):=\sqrt{(b_1+t)^2+(b_2-t)^2}+\sqrt{b_3^2+t^2},
\qquad t\in\mathbb{R}.
\end{equation}

Since the map $(x_1,x_2,x_3)\mapsto x_1+x_2$ is continuous and $0+1>0$, we may choose a neighborhood $U$ of $b_0 = (0,1,0)^T$ such that $b_1+b_2>0$ for every $b=(b_1,b_2,b_3)^T\in U$. Consequently, 
\begin{equation*}
    \frac{(b_1+b_2)^2}{\big[(b_1+t)^2+(b_2-t)^2\big]^{3/2}}>0 \,\text{ for each }\, b\in U.
\end{equation*}
We claim that $\varphi_b$ is strictly convex for every $b\in U$. Indeed, for such $b$, define
\[
g_b(t):=\sqrt{(b_1+t)^2+(b_2-t)^2},\quad h_b (t) := \sqrt{b_3^2 + t^2}.
\]
A direct computation gives
\begin{equation}\label{second deri}
g_b''(t)=\frac{(b_1+b_2)^2}{\big[(b_1+t)^2+(b_2-t)^2\big]^{3/2}}>0  \quad \forall\, t\in \R, 
\end{equation}
which yields the strict convexity of $g_b$ on $\R$. Since $g_b$ is strictly convex and $h_b$ is convex, the function $\varphi_b$ is also strictly convex on $\R$. Because $\varphi_b$ is continuous, coercive, and strictly convex, it admits a unique minimizer $t_b$. Consequently, $S(b)=\{x_b (t_b)\}$ for every $b\in U$, which implies that the solution mapping $S(\cdot)$ is locally single-valued around $b_0$. 
\end{Example} 
\begin{Remark}\label{Degree of Convergence}
{\rm In the above example, we have $S(b)=\{x_b(t)\}$ for $b$ locally around $b_0$. Combined with Proposition~\ref{converge}, every sequence of solutions $\{x_b(t)\}$ of $P(b)$ converges to $x_0$ as $b \to b_0$. However, Example~\ref{ex:NoSR} also shows that $S(\cdot)$ fails to possess the isolated calmness property at $b_0$ for $x_0$, i.e., the rate of convergence is less than $1$. It is therefore natural to investigate the actual rate at which $\{x_b\}$ converges to $x_0$. 

Consider the neighborhood $U$ of $b_0$ from Example~\ref{ex:NoSR}. By choosing a sequence $\{\delta_k\}$ converging to $0$, we have $b_k:=(\delta_k,1+\delta_k,\delta_k)^T$ belongs to $U$ for large enough $k\in \N$. Taking the derivative of $\varphi_{b_k}$,
\begin{equation}\label{eq:dphi}
    \varphi'_{b_k}(t)=\frac{2t-1}{\sqrt{(\delta_k+t)^2+(1+\delta_k-t)^2}}+\frac{t}{\sqrt{\delta_k^2+t^2}}.
\end{equation}
Let $t=0$ and $t=\frac{1}{2}$, it follows from \eqref{eq:dphi} that
\[
\varphi'_{b_k}(0)=\frac{-1}{\sqrt{\delta_k^2+(1+\delta_k)^2}}<0 \quad \mbox{and}\quad \varphi'_{b_k}\left(\dfrac{1}{2}\right)=\frac{1}{2\sqrt{\delta_k^2+\frac{1}{4}}}>0.
\]
Combining this and the fact that $\varphi'_{b_k}(t)=0$ admits a unique solution $t_k$ yields $t_k \in (0,1/2)$. Thus, there exists a subsequence $\{t_{k_m}\}_{m\in \N}$ of $\{t_k\}$ converging to some $t_0\in [0,1/2]$. 
%%%%%%%%%%%%%%
%Since t_0 may be 0 (which actually is), passing to the limits as m\to \infty in (3.52) may lead to indeterminate form 0/0 on the RHS
%%%%%%%%%%%%%%
% Hence, let $m\to \infty$ in \eqref{eq:dphi} and using $\varphi'_{b_{k_m}}(t_{k_m})=0$, we obtain
% \[
% \frac{2t_0-1}{\sqrt{t_0^{2}+(1-t_0)^{2}}}=-1.
% \]
% As $0\le t_0\le \frac{1}{2}$, the above equation is equivalent to $2t_0^2-2t_0=0$, i.e., $t_0=0$.   
Since $\varphi'_{b_{k}}(t_{k})=0$, we get from \eqref{eq:dphi} the equations
\begin{equation*}
    \dfrac{1-2t_{k}}{\sqrt{(\delta_{k}+t_{k})^2+(1+\delta_{k}-t_{k})^2}} = \dfrac{t_{k}}{\sqrt{\delta_{k}^2+t_{k}^2}}\qquad (k \text{ large}),
\end{equation*}
or equivalently,
\begin{equation}\label{eq:delta+t}
    (1-2t_{k})^2 \big(\delta_{k}^2 + t_{k}^2\big) = t^2_{k} \big[(\delta_{k}+t_{k})^2+(1+\delta_{k}-t_{k})^2\big] \qquad  (k\text{ large}).
\end{equation}
Taking the limits of both sides of \eqref{eq:delta+t} as $m\to \infty$, we arrive at $t_0=0$. Since $\{t_{k_m}\}$ was taken arbitrarily, it follows that $t_k\to 0$ as $k\to \infty$. Moreover, \eqref{eq:delta+t} also yields
\begin{equation}\label{eq:delta/t}
    (1-2t_{k})^2\frac{\delta_{k}^2}{t_{k}^2}=(\delta_{k}+t_{k})^2+(1+\delta_{k}-t_{k})^2-(1-2t_{k})^2\qquad (k \text{ large}).
\end{equation}
As $t_{k}\to 0$, the above equation helps us to derive that $\dfrac{\delta_k}{t_k}\to 0$ as $k\to \infty$.
Additionally, it is equivalent from \eqref{eq:delta/t} that
\begin{eqnarray}
    \begin{aligned}
        (1-2t_k)^2\delta_k^2&=t_k^2\big[(\delta_k+t_k)^2+(1+\delta_k-t_k)^2-(1-2t_k)^2\big]\\
        &=2t_k^2\big(\delta_k+\delta_k^2+t_k-t_k^2\big)\\
        &=2t_k^3\left[\frac{\delta_k}{t_k}(1+\delta_k)+1-t_k\right].
    \end{aligned}
\end{eqnarray}
Therefore,
\[
t_k=\sqrt[3]{\frac{(1-2t_k)^2\delta_k^2}{2\left[\dfrac{\delta_k}{t_k}(1+\delta_k)+1-t_k\right]}}\approx \mathcal{O}(\delta_k^{2/3}).
\]
Measuring the distance between the solutions $x_k$ of $P(b_k)$ and the original solution $x_0$ gives us that
\begin{eqnarray}
\begin{aligned}
    \|x_k-x_0\|
    &=\sqrt{(\delta_k+t_k)^2+(\delta_k-t_k)^2+\delta_k^2+t_k^2}\\
    &\approx \sqrt{\big[\mathcal{O} (\delta_k^{2/3})+\delta_k\big]^2+\big[\mathcal{O} (\delta_k^{2/3})-\delta_k\big]^2+\delta_k^2+\mathcal{O} (\delta_k^{2/3})^2}\approx \mathcal{O} (\delta_k^{2/3}).
\end{aligned}
\end{eqnarray}
This shows that the rate of convergence of the solutions $x_k$ of $P(b_k)$ to $x_0$ is less than $2/3$.}
\end{Remark}

\section{Conclusion}

This paper provides geometric characterizations of the uniqueness and isolated calmness properties for convex regularized linear inverse problems. The main results show that uniqueness can be described by a radial-cone condition involving pre-dual certificates, while isolated calmness is characterized by a tangent-cone condition. For convex piecewise linear-quadratic regularizers, these two properties coincide, yielding tractable verification procedures, including linear programming tests for analysis $\ell_1$-regularization. We provided a counterexample in a non-polyhedral setting showing that local single-valuedness alone is not sufficient for the isolated calmness property.

\bibliographystyle{abbrv}
\bibliography{bio}

\end{document}